\documentclass[12pt,a4paper]{article}
\usepackage{mathrsfs}
\usepackage{indentfirst}
\usepackage{amsmath,amssymb,amsfonts,amsthm,graphics,graphicx}
\usepackage{makeidx}
\usepackage{color}

\newtheorem{theorem}{Theorem}

\newtheorem{lemma}{Lemma}

\newtheorem{corollary}{Corollary}

\newtheorem{observation}{Observation}

\begin{document}
\title{\Large\bf On extremal graphs with at most two internally disjoint
Steiner trees connecting any three vertices\footnote{Supported by
NSFC No.11071130 and the ``973" project.}}
\author{\small Hengzhe Li, Xueliang Li, Yaping Mao
\\
\small Center for Combinatorics and LPMC-TJKLC
\\
\small Nankai University, Tianjin 300071, China
\\
\small lhz2010@mail.nankai.edu.cn; lxl@nankai.edu.cn;
maoyaping@ymail.com.}
\date{}
\maketitle
\begin{abstract}
The problem of determining the smallest number of edges,
$h(n;\overline{\kappa}\geq r)$, which guarantees that any graph with
$n$ vertices and $h(n;\overline{\kappa}\geq r)$ edges will contain a
pair of vertices joined by $r$ internally disjoint paths was posed
by Erd\"{o}s and Gallai. Bollob\'{a}s considered the problem of
determining the largest number of edges $f(n;\overline{\kappa}\leq
\ell)$ for graphs with $n$ vertices and local connectivity at most
$\ell$. One can see that $f(n;\overline{\kappa}\leq \ell)=
h(n;\overline{\kappa}\geq \ell+1)-1$. These two problems had
received a wide attention of many researchers in the last few
decades. In the above problems, only pairs of vertices connected by
internally disjoint paths are considered. In this paper, we study
the number of internally disjoint Steiner trees connecting sets of
vertices with cardinality at least $3$.

\end{abstract}

\section{Introduction}

All graphs considered in this paper are undirected, finite and
simple. We refer to book \cite{Bondy} for graph theoretical notation
and terminology not described here. We call the number of vertices
in a graph as the {\it order} of the graph and the number of edges
of it as its {\it size}. For two distinct vertices in a connected
graph $G$, we can connect them by a path. Two paths are called
\emph{internally disjoint} if they have no common vertex except the
end vertices. For any two distinct vertices $x$ and $y$ in $G$, the
\emph{local connectivity} $\kappa_{G}(x,y)$ is the maximum number of
internally disjoint paths connecting $x$ and $y$. Then
$\min\{\kappa_{G}(x,y)|x,y\in V(G),x\neq y\}$ is usually the
connectivity of $G$. In contrast to this parameter,
$\overline{\kappa}(G)=\max\{\kappa_{G}(x,y)|x,y\in V(G),x\neq y\}$,
introduced by Bollob\'{a}s, is called the \emph{maximum local
connectivity} of $G$. The problem of determining the smallest number
of edges, $h(n;\overline{\kappa}\geq r)$, which guarantees that any
graph with $n$ vertices and $h(n;\overline{\kappa}\geq r)$ edges
will contain a pair of vertices joined by $r$ internally disjoint
paths was posed by Erd\"{o}s and Gallai, see \cite{Bartfai} for
details.

Bollob\'{a}s \cite{Bollobas1} considered the problem of determining
the largest number of edges, $f(n;\overline{\kappa}\leq \ell)$, for
graphs with $n$ vertices and local connectivity at most $\ell$.
Actually, $f(n;\overline{\kappa}\leq \ell)=\max\{e(G)| |V(G)|=n\
and\ \overline{\kappa}(G)\leq \ell\}$. Motivated by determining the
precise value of $f(n;\overline{\kappa}\leq \ell)$, this problem has
obtained wide attention and many results have been worked out, see
\cite{Bollobas1, Bollobas2, Bollobas3, Thomassen, Leonard1,
Leonard2, Leonard3, Mader1, Mader2}. One can see that
$h(n;\overline{\kappa}\geq \ell+1)=f(n;\overline{\kappa}\leq
\ell)+1$.

For $\overline{\kappa}(G)\leq \ell$, it was showed that
$f(n;\overline{\kappa}\leq \ell)\geq
\lfloor\frac{\ell+1}{2}(n-1)\rfloor$. Since
$f(n;\overline{\kappa}\leq
\ell)=\lfloor\frac{\ell+1}{2}(n-1)\rfloor$ for $\ell=2,3$,
Bollob\'{a}s and Erd\"{o}s conjectured that the equality holds, but
this was disproved by Leonard \cite{Leonard1} for $\ell=4$, and
later Mader \cite{Mader1} constructed graphs disproving it for every
$\ell \geq 4$.

For a graph $G(V,E)$ and a set $S\subseteq V$ of at least two
vertices, \emph{an $S$-Steiner tree} or \emph{an Steiner tree
connecting $S$} (or simply, \emph{an $S$-tree}) is a such subgraph
$T(V',E')$ of $G$ that is a tree with $S\subseteq V'$. Two Steiner
trees $T$ and $T'$ connecting $S$ are \emph{internally disjoint} if
$E(T)\cap E(T')=\varnothing$ and $V(T)\cap V(T')=S$. For $S\subseteq
V(G)$, the \emph{generalized local connectivity} $\kappa_{G}(S)$ is
the maximum number of internally disjoint trees connecting $S$ in
$G$. The \emph{generalized connectivity}, introduced by Chartrand et
al. in 1984 \cite{Chartrand}, is defined as
$\kappa_k(G)=min\{\kappa(S)|S\subseteq V(G),|S|=k\}$. There have
been many results on the generalized connectivity, see
\cite{LLSun,LL, LLL, LLZ}. Similar to the classical maximal local
connectivity, we introduce another parameter
$\overline{\kappa}_k(G)=max\{\kappa(S)|S\subseteq V(G),|S|=k\}$,
which is called the {\it maximum generalized local connectivity} of
$G$. It is easy to check that
$0\leq\overline{\kappa}_k(G)\leq\overline{\kappa}_k(K_n)\leq
n-k+\lfloor k/2\rfloor$ for a graph $G$.

In this paper, we mainly study the problem of determining the
largest number of edges, $f(n;\overline{\kappa}_k\leq \ell)$, for
graphs with $n$ vertices and maximum generalized local connectivity
at most $\ell$, where $0\leq\ell\leq n-k+\lfloor k/2\rfloor$. That
is, $f(n;\overline{\kappa}_k\leq \ell)=\max\{e(G)| |V(G)|=n\ and\
\overline{\kappa}_k(G)\leq \ell\}$. We also study the smallest
number of edges, $h(n;\overline{\kappa}_k\geq r)$, which guarantees
that any graph with $n$ vertices and $h(n;\overline{\kappa}_k\geq
r)$ edges will contain a set $S$ of $k$ vertices such that there are
$r$ internally disjoint $S$-trees, where $0\leq r\leq n-k+\lfloor
k/2\rfloor$. It is not difficult to see that
$h(n;\overline{\kappa}_k\geq \ell+1)= f(n;\overline{\kappa}_k\leq
\ell)+1$ for $0\leq \ell\leq n-k+\lfloor k/2\rfloor-1$. For $k=3$
and $\ell=2$, we prove that $f(n;\overline{\kappa}_3\leq 2)=2n-3$
for $n\geq 3$ and $n\neq 4$, and $f(n;\overline{\kappa}_3\leq
2)=2n-2$ for $n=4$. Furthermore, we characterize the graphs
attaining these values. For $k=3$ and a general $\ell$, we construct
some graphs to show that $f(n;\overline{\kappa}_3\leq \ell)\geq
\frac{\ell+2}{2}(n-2)+\frac{1}{2}$ for both $n$ and $\ell$ odd, and
$f(n;\overline{\kappa}_3\leq \ell)\geq \frac{\ell+2}{2}(n-2)+1$
otherwise.

\section{Some basic results}

As usual, the \emph{union} of two graphs $G$ and $H$ is the graph,
denoted by $G\cup H$, with vertex set $V(G)\cup V(H)$ and edge set
$E(G)\cup E(H)$. The disjoint union of $k$ copies of the same graph
$G$ is denoted by $kG$. The \emph{join} $G\vee H$ of two disjoint
graphs $G$ and $H$ is obtained from $G\cup H$ by joining each vertex
of $G$ to every vertex of $H$.

In this section, we first introduce a graph operation and two graph
classes.

Let $H$ be a connected graph, and $u$ a vertex of $H$. We define the
{\it attaching operation at the vertex $u$} on $H$ as follows: (1)
identifying $u$ and a vertex of a $K_4$; (2) $u$ is attached with
only one $K_4$. The vertex $u$ is called \emph{an attaching vertex}.

\begin{figure}[h,t,b,p]
\begin{center}
\scalebox{0.8}[0.8]{\includegraphics{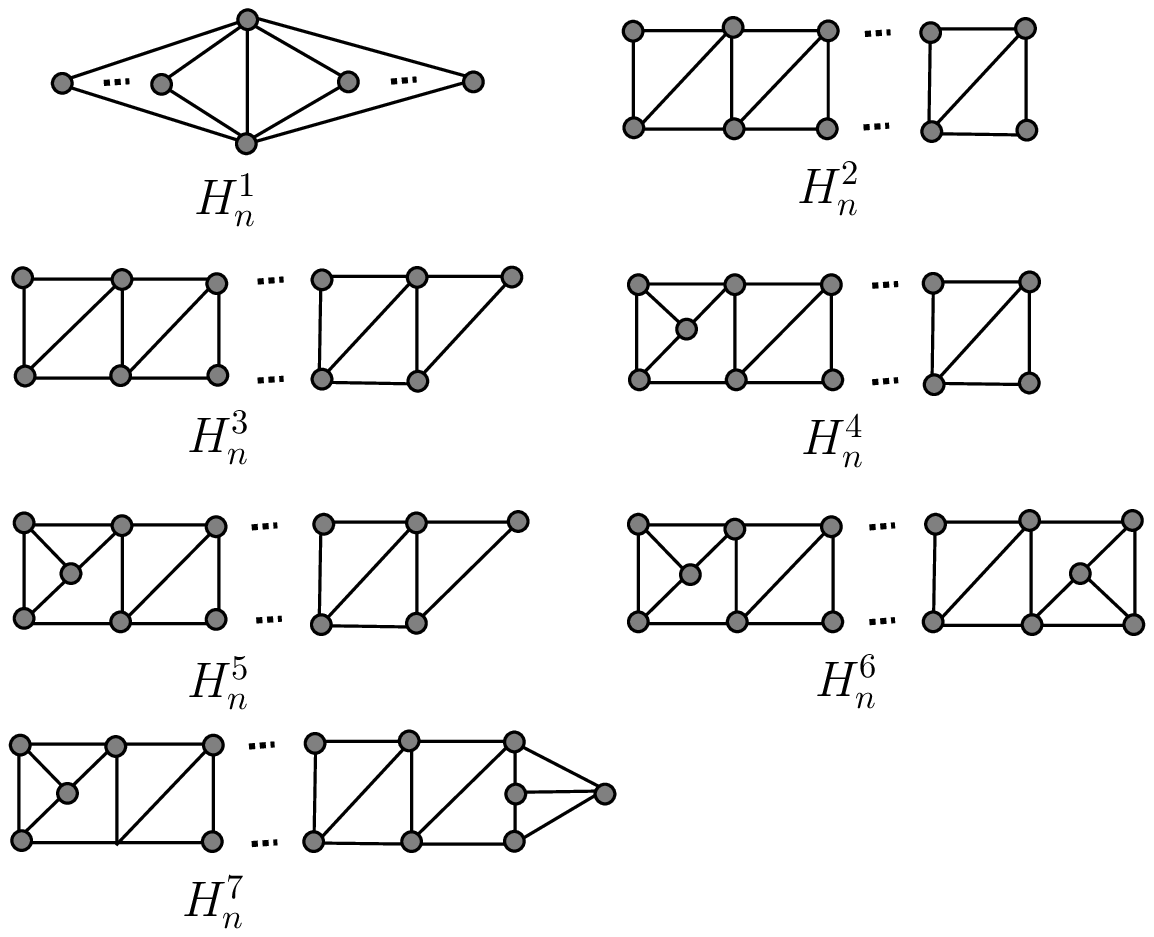}}\\[20pt]
 Figure 1. The graph class $\mathcal{G}_{n}$.
\end{center}\label{fig1}
\end{figure}

Now, we introduce two new graph classes. For $n\geq 5$,
$\mathcal{G}_{n}=\{H_n^1,H_n^2,H_n^3,H_n^4,\linebreak[2]
H_n^5,H_n^6,H_n^7\}$ is a class of graphs of order $n$ (see Figure
$1$ for details). Let $\mathcal{H}_n^{i} \ (1\leq i\leq 7)$ be the
class of graphs, each of them is obtained from a graph $H_r^{i}$ by
the attaching operation at some vertices of degree $2$ on $H_r^{i}$,
where $3\leq r\leq n$ and $1\leq i\leq 7$ (note that $H_n^{i}\in
\mathcal{H}_n^{i}$). $\mathcal{G}_{n}^*$ is another class of graphs
that contains $\mathcal{G}_{n}$, given as follows:
$\mathcal{G}_{3}^*=\{K_3\}$, $\mathcal{G}_{4}^*=\{K_4\}$,
$\mathcal{G}_{5}^*=\{G_1\}\cup (\bigcup_{i=1}^7 \mathcal{H}_5^{i})$,
$\mathcal{G}_{6}^*=\{G_3,G_4\}\cup (\bigcup_{i=1}^7
\mathcal{H}_6^{i})$, $\mathcal{G}_{7}^*=\bigcup_{i=1}^7
\mathcal{H}_7^{i}$, $\mathcal{G}_{8}^*=\{G_2\}\cup (\bigcup_{i=1}^7
\mathcal{H}_8^{i})$, $\mathcal{G}_{n}^*=\bigcup_{i=1}^7
\mathcal{H}_n^{i}$ for $n\geq 9$ (see Figure $2$ for details).

\begin{figure}[h,t,b,p]
\begin{center}
\scalebox{0.8}[0.8]{\includegraphics{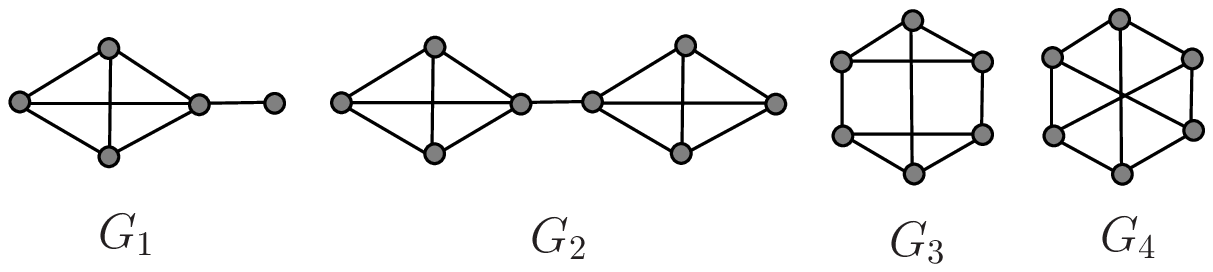}}\\[20pt]

Figure 2. Some graphs in $\mathcal{G}_{n}^*$.
\end{center}\label{fig2}
\end{figure}

It is easy to see that the following three observations hold.

\begin{observation}\label{obs1}
Let $G$ and $H$ be two connected graphs, and $H'$ be a subdivision
of $H$. If $H'$ is a subgraph of $G$ and $\overline{\kappa}_3(H)\geq
3$, then $\overline{\kappa}_3(G)\geq 3$.
\end{observation}

\begin{observation}\label{obs2}
Let $H$ be a graph, $u$ and $v$ be two vertices in $H$, and $G$ be a
graph obtained from $H$ by attaching a $K_4$ at $u$. If there are
three internally disjoint paths between $u$ and $v$ in $H$, then
$\overline{\kappa}_3(G)\geq 3$.
\end{observation}

\begin{observation}\label{obs3}
For each graph in Figure $3$, $\overline{\kappa}_3(G)\geq 3$.
\end{observation}

\begin{figure}[h,t,b,p]
\begin{center}
\scalebox{0.8}[0.8]{\includegraphics{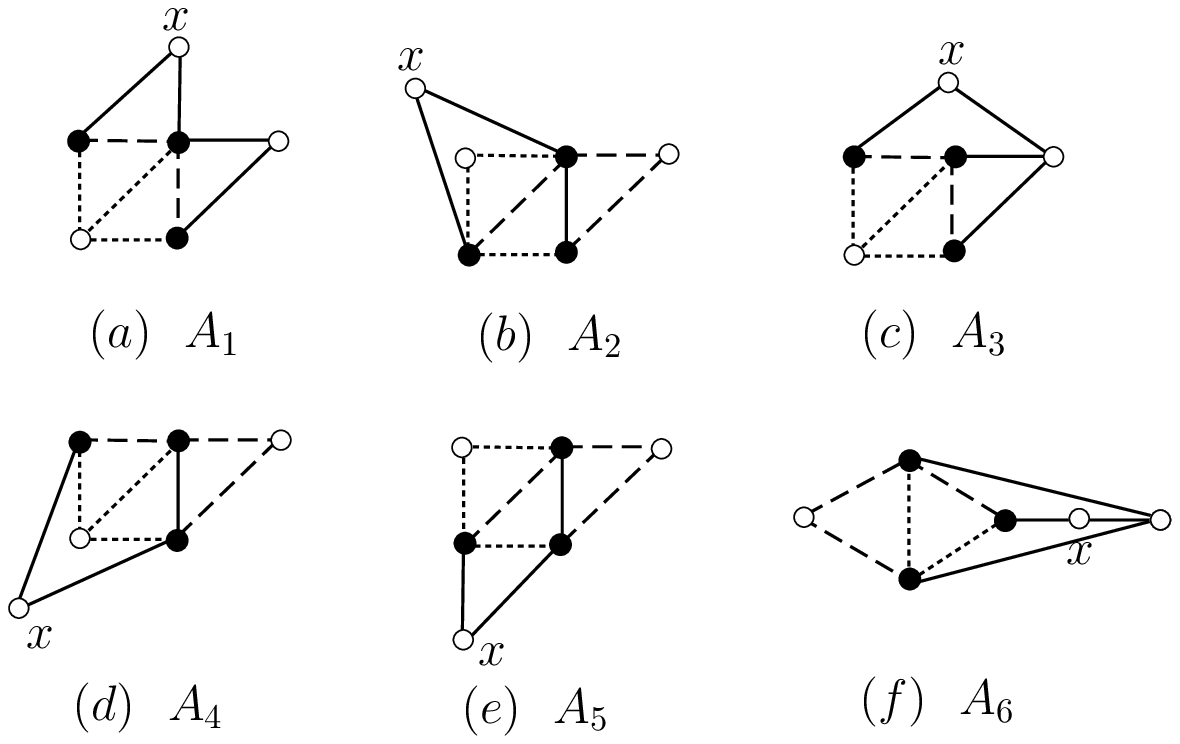}}\\
Figure 3. Graphs obtained from $H_5^3$ by adding a vertex of degree
$2$.
\end{center}
\end{figure}

\begin{lemma}\label{lem1}
Let $G$ be a graph containing a clique $K_4$. If there exists a path
connecting two vertices of $K_4$ in $G\setminus E[K_4]$, then
$\overline{\kappa}_3(G)\geq 3$.
\end{lemma}

\begin{proof}
Let $K_4$ be a complete subgraph of $G$ with vertex set
$\{u_1,\cdots,u_4\}$, and $P$ be a path connecting $u_1$ and $u_2$
in $G\setminus E[K_4]$. It suffices to show that there exists a set
$S$ such that $\kappa(S)\geq 3$.  Choose $S=\{u_1,u_2,u_3\}$,
clearly, $T_1=u_1u_2\cup u_1u_3$ and $T_2=u_4u_1\cup u_4u_2\cup
u_4u_3$ and $T_3=P\cup u_2u_3$ form three internally disjoint
$S$-trees. Thus, $\overline{\kappa}_3(G)\geq 3$.
\end{proof}

Similarly, the following lemma holds.

\begin{lemma}\label{lem2}
Let $G$ be a graph obtained from $H_5^4$ by adding a vertex $x$ and
two edges $xy, xz$, where $y,z\in V(H_5^4)$(see Figure $4$). Then
$\overline{\kappa}_3(G)\geq 3$ or $G=H_6^{5}$.
\end{lemma}

\begin{figure}[h,t,b,p]
\begin{center}
\scalebox{0.8}[0.8]{\includegraphics{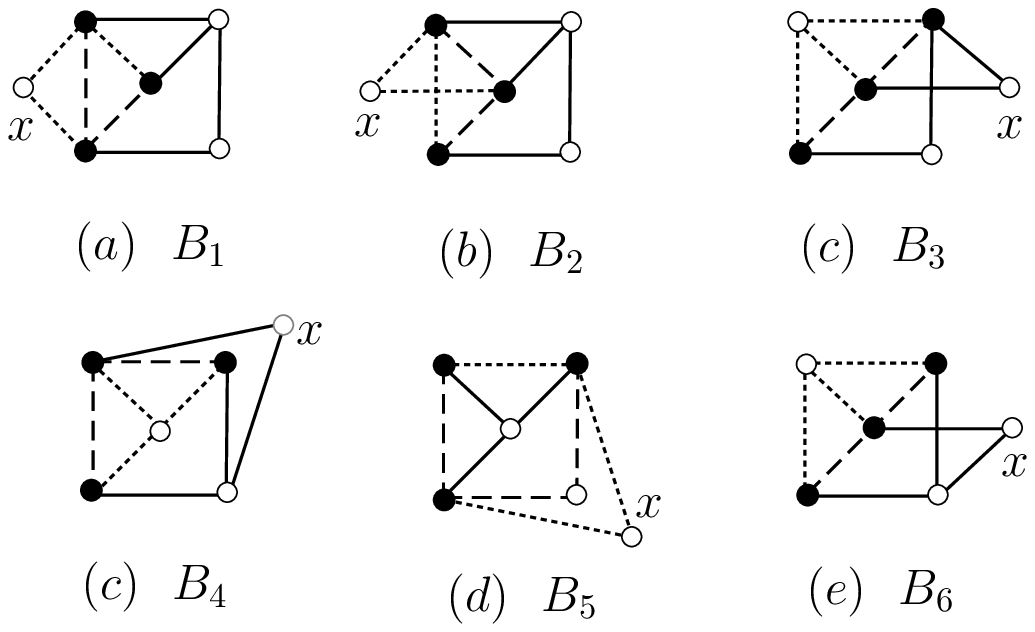}}\\
Figure 4. Graphs obtained from $H_5^4$ by adding a vertex of degree
$2$.
\end{center}
\end{figure}

\begin{lemma}\label{lem3}
For any connected graph $G$ with order $5$ and size $8$,
$\overline{\kappa}_3(G)\geq 3$.
\end{lemma}

\begin{proof}
We claim that $2\leq \delta(G)\leq 3$. In fact, if $\delta(G)=1$,
without loss of generality, let $d(x)=1$, then $|V(G-x)|=4$ and
$e(G-x)=7$, a contradiction. If $\delta(G)\geq 4$, then
$16=2e(G)\geq 5\delta \geq 20$, a contradiction.

If $\delta(G)=2$, without loss of generality, let $d(x)=2$, then
$|V(G-x)|=4$ and $e(G-x)=6$, which implies that $G-x$ is a clique of
order $4$. From Lemma \ref{lem1}, $\overline{\kappa}_3(G)\geq 3$. So
we suppose that $\delta(G)=3$. Since $|V(G)|=5$, $\Delta(G)\leq 4$.
Since $\frac{2e(G)}{|V(G)|}=\frac{16}{5}$, there exists a vertex $x$
in $G$ such that $d(x)=4$. Set $N_G(x)=\{u_1,u_2,u_3,u_4\}$. Since
$\delta(G-x)\geq 2$ and $e(G-x)=4$, $G-x$ is a cycle of order $4$.
Then $G$ is a wheel of order $5$ and the trees $T_1=xu_2\cup xu_4$
and $T_2=u_3x\cup u_3u_2\cup u_3u_4$ and $T_3=u_1x\cup u_1u_4\cup
u_1u_2$ form $3$ internally disjoint $\{x,u_2,u_4\}$-trees, namely,
$\overline{\kappa}_3(G)\geq 3$.\end{proof}

\begin{lemma}\label{lem4}
For any connected graph $G$ of order $5$ and size $7$,
$\overline{\kappa}_3(G)\leq 2$ and $G\in \{G_1,H_5^1,H_5^3,H_5^4\}$.
\end{lemma}

\begin{proof}
For each $S\subseteq V(G)$ with $|S|=3$, a tree with two edges
connecting $S$ is called {\it Type $I$}, and the others with at
least $3$ edges are called {\it Type $II$}. One can see that three
internally disjoint trees connecting $S$ will use at least $8$ edges
since we only have one tree of Type $I$. So if $G$ is a connected
graph of order $5$ and size $7$, then $\overline{\kappa}_3(G)\leq
2$.

Suppose that $\delta(G)\geq 3$. Then $14=2e(G)\geq 5\delta \geq 15$,
a contradiction. Thus, $\delta(G)\leq 2$. If $\delta(G)=1$, without
loss of generality, let $d(x)=1$, then $|V(G-x)|=4$ and $e(G-x)=6$,
which implies that $G-x$ is a clique of order $4$. Then $G=G_1$ (see
Figure $2$).

\begin{figure}[h,t,b,p]
\begin{center}
\scalebox{0.8}[0.8]{\includegraphics{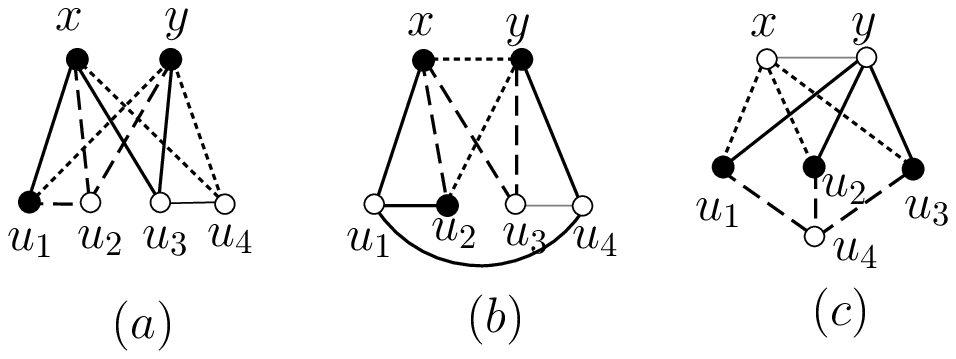}}\\

Figure 5. Graphs for Lemma \ref{lem5}.
\end{center}
\end{figure}

If $\delta(G)=2$, without loss of generality, let $d(x)=2$, then
$|V(G-x)|=4$ and $e(G-x)=5$, which implies that $G-x$ is graph
obtained from $K_4$ by deleting an edge. Thus, $G\in
\{H_5^1,H_5^3,H_5^4\}$ (see Figure 1).
\end{proof}

\begin{lemma}\label{lem5}
For any connected graph $G$ with order $6$ and size $10$,
$\overline{\kappa}_3(G)\geq 3$.
\end{lemma}

\begin{proof}
If there exists a vertex $x\in V(G)$ such that $d(x)\leq 2$, then
$|V(G-x)|=5$ and $e(G-x)\geq 8$. From Lemma \ref{lem3},
$\overline{\kappa}_3(G-x)\geq 3$, which results in
$\overline{\kappa}_3(G)\geq 3$.

Now we assume that $\delta(G)\geq 3$. If there exists a vertex $x\in
V(G)$ such that $d(x)=5$, then $|V(G-x)|=5$ and $e(G-x)=5$. Since
$\delta(G-x)\geq 2$, $G-x$ is a cycle of order $5$, which implies
that $G$ is wheel of order $6$. Clearly, $\overline{\kappa}_3(G)\geq
3$. So we can assume that $\Delta(G)\leq 4$. Let $t$ be the number
of vertices of degree $4$ in $G$. Since $20=2e(G)=4t+3(6-t)$, $t=2$,
namely, there exist two vertices $x,y\in V(G)$ such that
$d(x)=d(y)=4$.

If $xy\notin E(G)$, then $G$ must be the graph shown in Figure 5
$(a)$ since $\delta(G)\geq 3$. Then the trees $T_1=u_2x\cup u_2y\cup
u_2u_1$ and $T_2=u_1x\cup xu_3\cup u_3y$ and $T_3=u_1y\cup yu_4\cup
u_4x$ form three $\{x,y,u_1\}$-trees, namely,
$\overline{\kappa}_3(G)\geq 3$.

If $xy\in E(G)$ and $N_{G-xy}(x)\neq N_{G-xy}(y)$, then $G$ must be
the graph shown in Figure 5 $(b)$ since $\delta(G)\geq 3$. Then the
trees $T_1=u_2x\cup xu_3\cup u_3y$ and $T_2=yx\cup yu_2$ and
$T_3=u_1x\cup u_1u_2\cup u_1u_4\cup u_4y$ form three
$\{x,y,u_2\}$-trees, namely, $\overline{\kappa}_3(G)\geq 3$.

If $xy\in E(G)$ and $N_{G-xy}(x)=N_{G-xy}(y)$, then $G$ must be the
graph shown in Figure 5 $(c)$ since $\delta(G)\geq 3$. Then the
trees $T_1=xu_1\cup xu_2\cup xu_3$ and $T_2=yu_1\cup yu_2\cup yu_3$
and $T_3=u_4u_1\cup u_4u_2\cup u_4u_3$ form three
$\{u_1,u_2,u_3\}$-trees, namely, $\overline{\kappa}_3(G)\geq
3$.\end{proof}

\begin{lemma}\label{lem6}
Let $G$ be a connected graph of order $6$ and size $9$. If
$\overline{\kappa}_3(G)\leq 2$, then $G\in \{G_3,G_4\}$ or $G\in
\{H_6^1,H_6^2,H_6^5\}$ or $G\in \mathcal{H}_6^{3}$.
\end{lemma}

\begin{proof}
We claim that $2\leq \delta(G)\leq 3$. Suppose that $\delta(G)\geq
4$. Then $18=2e(G)\geq 6\delta \geq 24$, a contradiction. Suppose
that $\delta(G)=1$, without loss of generality, let $d(x)=1$, then
$|V(G-x)|=5$ and $e(G-x)=8$. From Lemma \ref{lem3},
$\overline{\kappa}_3(G-x)\geq 3$. Clearly,
$\overline{\kappa}_3(G)\geq 3$ by Observation \ref{obs1}.

If $\delta(G)=3$, then $G$ is $3$-regular. It is easy to check that
$G=G_3$ or $G=G_4$. In the following, we assume that $\delta(G)=2$.
Without loss of generality, set $d(x)=2$, then $|V(G-x)|=5$ and
$e(G-x)=7$, which implies that $G-x=G_1$ or $G-x\in
\{H_5^1,H_5^3,H_5^4\}$ by Lemma \ref{lem4}.

If $G-x=G_1$, then $G\in \mathcal{H}_6^{3}$. If $G-x=H_5^1$, then
$G=H_6^1$ or $G=A_2$ or $G=A_6$ (see Figure 3), which results in
$G=H_6^1$. If $G-x=H_5^3$, then $G=H_6^2$ or $G\in
\{A_1,A_2,A_3,A_4,A_5\}$, which implies that $G=H_6^2$ by
Observation \ref{obs3}. If $G-x=H_5^4$, then $G=H_6^5$ or $G\in
\{B_1,B_2,B_3,B_4,B_5,B_6\}$, which implies that $G=H_6^5$ by Lemma
\ref{lem2}.
\end{proof}

\section{Main results}

In this section, we give our main results. We first need some more
lemmas. In Lemma \ref{lem3} through Lemma \ref{lem6}, we have dealt
with the cases $n\leq 6$. Now we assume that $n\geq 7$.

\begin{lemma}\label{lem7}
Let $G'$ be a graph obtained from $G$ by deleting a vertex of degree
$2$. If $G'\in \mathcal {G}_{n-1}^*$($n\geq 7$), then $G\in \mathcal
{G}_{n}^*$ or $\overline{\kappa}_3(G)\geq 3$.
\end{lemma}

\begin{proof}
Let $x$ be the deleted vertex of degree $2$ in $G$. Since $n\geq 7$,
$G'\notin \{K_3,K_4,G_1\}$. From Observation \ref{obs2} and Lemma
\ref{lem1}, if $G'\in \{G_2,G_3,G_4\}$, then we can check that $G\in
\mathcal{H}_9^{3}$ or $\overline{\kappa}_3(G)\geq 3$. From now on,
we consider $G'\in \mathcal {G}_{n-1}^*\setminus \{G_2,G_3,G_4\}$.

\textbf{Case 1. } $G'\in \mathcal{H}_{n-1}^{1}$.

First we consider the case that there is no $K_4$ in $G'$. Thus,
$G'=H_{n-1}^{1}$. Since $n\geq 7$, $G=H_{n}^{1}\in
\mathcal{H}_{n}^{1}$ or $G$ must contain an $A_2$ or $A_6$ as its
subgraph, which implies that $G\in \mathcal {G}_{n}^{*}$ or
$\overline{\kappa}_3(G)\geq 3$ by Observation \ref{obs1}.

Next we consider the case that there exists at least one $K_4$ in
$G'$. For each $K_4$, if $N_G(x)\cap (K_4\setminus y)\neq
\varnothing$, then we have $\overline{\kappa}_3(G)\geq 3$ by Lemma
$1$, where $y$ is an attaching vertex in $G'$. Suppose that
$N_G(x)\cap (K_4\setminus y)=\varnothing$ for all $K_4\subseteq G'$.
Clearly, we can consider the graph $G'\in \mathcal{H}_{n-1}^{1}$ as
the join of $K_2$ and $r$ isolated vertices, and then doing the
attaching operation at some vertices of degree $2$ on $K_2\vee
rK_1$. So, we consider $N(x)\subseteq K_2\vee rK_1 \ (r\geq 1)$. For
$r\geq 3$, it follows that $G\in \mathcal{H}_{n}^{1}$ or $G$
contains the graph $A_2$ or $A_6$ as its subgraph, which implies
that $G\in \mathcal {G}_{n}^{*}$ or $\overline{\kappa}_3(G)\geq 3$.

For $r=2$, from Lemma $1$, we only need to consider $N(x)\subseteq
V(K_2\vee 2K_1)$. By Observation \ref{obs2}, $G\in
\mathcal{H}_{11}^{1}$ or $G\in \mathcal{H}_8^{1}$ or $G\in
\mathcal{H}_8^{3}$ or $\overline{\kappa}_3(G)\geq 3$. For $r=1$,
$K_2\vee K_1$ is a triangle and $G'$ is a graph obtained from this
triangle by the attaching operation at two or three vertices of this
triangle since $n\geq 7$. Thus, from Observation \ref{obs2} and
Lemma \ref{lem1}, we can get $\overline{\kappa}_3(G)\geq 3$.

\textbf{Case 2. } $G'\in \mathcal{H}_{n-1}^{2}$ or $G'\in
\mathcal{H}_{n-1}^{3}$.

We only prove the conclusion for $G'\in \mathcal{H}_{n-1}^{2}$, the
same can be showed for $G'\in \mathcal{H}_{n-1}^{3}$ similarly.
Without loss of generality, let $\mathcal{H}_{n-1}^{2}$ be the graph
class obtained from $H_{r}^{2}$ by the attaching operation at some
vertices of degree $2$ on $H_{r}^{2}$, where $r=n-1,n-4,n-7$. One
can see that $u_1$ and $v_{\frac{r}{2}}$ can be the attaching
vertices. From Lemma \ref{lem1}, we only need to consider the case
that $N_G(x)\subseteq H_{r}^{2}$. Set $N_G(x)=\{x_1,x_2\}$. Thus
$x_1,x_2\in V(H_{r}^{2})$.

If $d_{H_{r}^{2}}(x_1)=d_{H_{r}^{2}}(x_2)=2$, without loss of
generality, let $x_1=u_1$ and $x_2=v_{\frac{r}{2}}$, then neither
$u_1$ nor $v_{\frac{r}{2}}$ is an attaching vertex by Observation
\ref{obs2}. We can choose a path $P:=u_3u_4\cdots
u_{\frac{r}{2}}v_{\frac{r}{2}}xu_1$ connecting $u_1$ and $u_3$ in
$G\setminus \{u_2,v_1,v_2\}$. Thus, $G$ contains a subdivision of
$A_3$ as its subgraph (see Figures 3 and 6 $(a)$), which results in
$\overline{\kappa}_3(G)\geq 3$.

If $d_{H_{r}^{2}}(x_1)=2$ and $d_{H_{r}^{2}}(x_2)=3$, without loss
of generality, let $x_1=u_1$, then we can find a path connecting
$u_1$ and $u_3$ and obtain $\overline{\kappa}_3(G)\geq 3$ for
$x_2\in H_r\setminus \{u_2,v_1,v_2\}$. For $x_2=u_2$ and $x_2=v_2$,
$G$ contains an $A_1$ and $A_4$ as its subgraph, which implies
$\overline{\kappa}_3(G)\geq 3$. If $x_2=v_1$, then $G\in
\mathcal{H}_n^3$ and so $G\in \mathcal {G}_n^*$.

\begin{figure}[h,t,b,p]
\begin{center}
\scalebox{0.8}[0.8]{\includegraphics{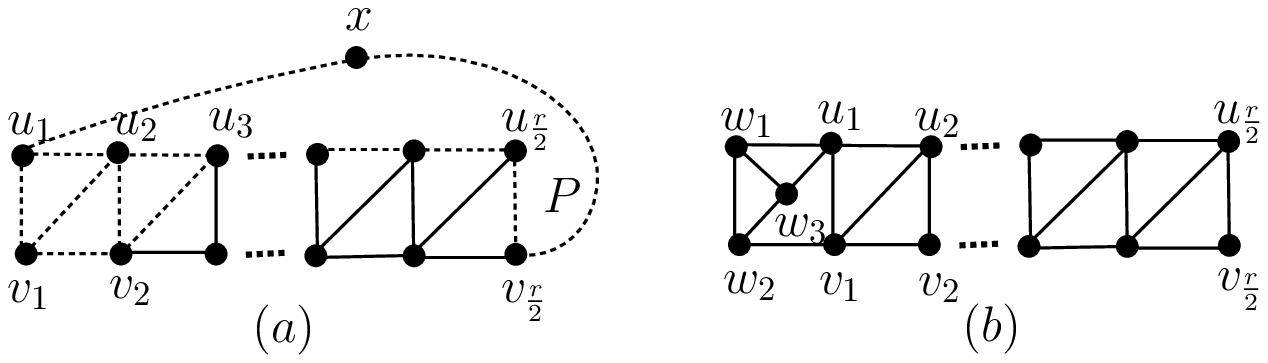}}\\
Figure 6. Graphs for Lemma \ref{lem7}.
\end{center}\label{fig6}
\end{figure}

For $3\leq d_{H_{r}^{2}}(x_i)\leq 4 \ (i=1,2)$, one can check that
$G$ contains a subdivision of one of $\{A_1,A_2,\cdots,A_5\}$, which
implies $\overline{\kappa}_3(G)\geq 3$.

\textbf{Case 3. } $G'\in \mathcal{H}_{n-1}^{4}$ or $G'\in
\mathcal{H}_{n-1}^{5}$.

Note that only $v_{\frac{r}{2}}$ can be an attaching vertex in
$H_{r}^{4}$ (see Figure 6 $(b)$), where $r=n-1,n-4$. From Lemma
\ref{lem1}, we only need to consider $N(x)\subseteq H_{r}^{4}$. We
can consider $H_{r}^{4}$ as a graph obtained from $H_{5}^{4}$ and
$H_{r-3}^{2}$ by identifying one edge $u_1v_1$ in each of them.

If $N(x)\cap \{w_1,w_2,w_3\}\neq \varnothing$, then $G$ contains a
subdivision of one of $\{B_1,\cdots,B_6\}$ as its subgraph. So,
$\overline{\kappa}_3(G)\geq 3$ by Lemma \ref{lem2}. Now we can
assume that $N(x)\cap \{w_1,w_2,w_3\}=\varnothing$. For
$|\{u_1,v_1\}\cap N(x)|= 2$, $G$ contains an $A_2$ as its subgraph,
which results in $\overline{\kappa}_3(G)\geq 3$. For
$|\{u_{\frac{r}{2}},v_{\frac{r}{2}}\}\cap N(x)|= 2$, if
$v_{\frac{r}{2}}$ is not an attaching vertex in $H_{r}^{4}$, then
$G\in \mathcal{H}_{n}^{5}$; if $v_{\frac{r}{2}}$ is an attaching
vertex in $H_{r}^{4}$, then $\overline{\kappa}_3(G)\geq 3$ by
Observation \ref{obs2}. For the other cases, we can also check that
$\overline{\kappa}_3(G)\geq 3$.

\textbf{Case 4. } $G'\in \mathcal{H}_{n-1}^{6}$ or $G'\in
\mathcal{H}_{n-1}^{7}$.

From the above Case $2$ and Lemma \ref{lem2}, we can get
$\overline{\kappa}_3(G)\geq 3$ in this case.
\end{proof}

Similarly, we have the following lemma.

\begin{lemma}\label{lem8}
Let $G'$ be a graph obtained from $G$ by deleting a vertex of degree
$3$. If $G'\in \mathcal{G}_{n-1}^*$($n\geq 7$), then
$\overline{\kappa}_3(G)\geq 3$.
\end{lemma}

\begin{lemma}\label{lem9}
Let $G$ be a graph obtained from $G'$ by deleting an edge $e=x_1x_2$
and adding a vertex $x$ such that $N_G(x)=\{x_1,x_2,x_3\}$, where
$x_3\in V(G')\setminus \{x_1,x_2\}$. If $G'\in
\mathcal{G}_{n-1}^{*}$($n\geq 7$), then $G\in \mathcal{G}_{n}^{*}$
or $\overline{\kappa}_3(G)\geq 3$.
\end{lemma}

\begin{proof}

Since $n\geq 7$, $G'\notin \{K_3,K_4,G_1\}$. From Observation
\ref{obs2} and Lemma \ref{lem1}, if $G'\in \{G_2,G_3,G_4\}$, we can
easily check that $\overline{\kappa}_3(G)\geq 3$. Thus we consider
$G'\in \mathcal {G}_{n-1}^*\setminus \{G_2,G_3,G_4\}$.

We claim that if there exists a $K_4$ in $G'$ such that $e\in
E(K_4)$, then $\overline{\kappa}_3(G)\geq 3$. Let
$V(K_4)=\{u_1,u_2,u_3,u_4\}$. Without loss of generality, let
$x_1=u_2$ and $x_2=u_4$.

If $x_3\in V(K_4)$, then $x_3=u_1$ or $x_3=u_3$. It follows that
$\overline{\kappa}_3(G)\geq 3$ (see Figure 7 $(a)$). So we assume
that $x_3\notin V(K_4)$. From Lemma \ref{lem1}, if $x_3$ belongs to
another clique of order $4$ such that $x_3$ is not an attaching
vertex, then $\overline{\kappa}_3(G)\geq 3$. So, we only need to
consider $x_3\in H_r^i(1\leq i\leq 7)$. If neither $u_2$ nor $u_4$
is an attaching vertex, then $u_1$ or $u_3$ is an attaching vertex,
say $u_1$. Then there must exist a path $P$ connecting $x_3$ and
$u_1$ such that $u_2,u_3,u_4\notin V(P)$ since $H_r^i(1\leq i\leq
7)$ is connected. Then the trees $T_1=xu_2\cup xu_4\cup P$ and
$T_2=u_1u_2\cup u_1u_4$ and $T_3=u_3u_1\cup u_3u_2\cup u_3u_4$ form
three $\{u_1,u_2,u_4\}$-trees, namely, $\overline{\kappa}_3(G)\geq
3$ (see Figure 7 $(b)$).
\begin{figure}[h,t,b,p]
\begin{center}
\scalebox{0.9}[0.9]{\includegraphics{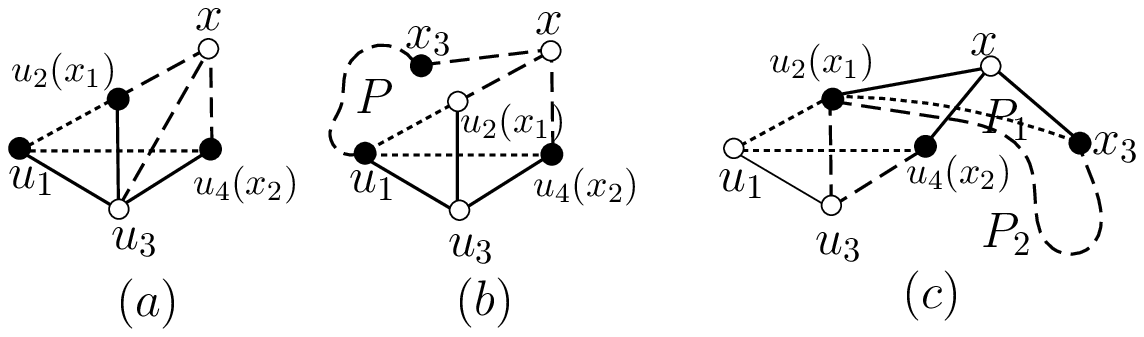}}\\
Figure 7. Graphs for the claim.
\end{center}
\end{figure}

Suppose that one of $\{u_2,u_4\}$ is an attaching vertex, say $u_2$.
Thus there must exist two paths $P_1$ and $P_2$ connecting $x_3$ and
$u_2$ in $H_r^i$ since $H_r^i$ is $2$-connected. Then the trees
$T_1=xu_2\cup xu_4\cup xx_3$ and $T_2=u_4u_1\cup u_1u_2\cup P_1$ and
$T_3=u_4u_3\cup u_3u_2\cup P_2$ form three internally disjoint
$\{u_2,u_3,x_3\}$-trees, namely, $\overline{\kappa}_3(G)\geq 3$ (see
Figure 7 $(c)$).

\begin{figure}[h,t,b,p]
\begin{center}
\scalebox{0.8}[0.8]{\includegraphics{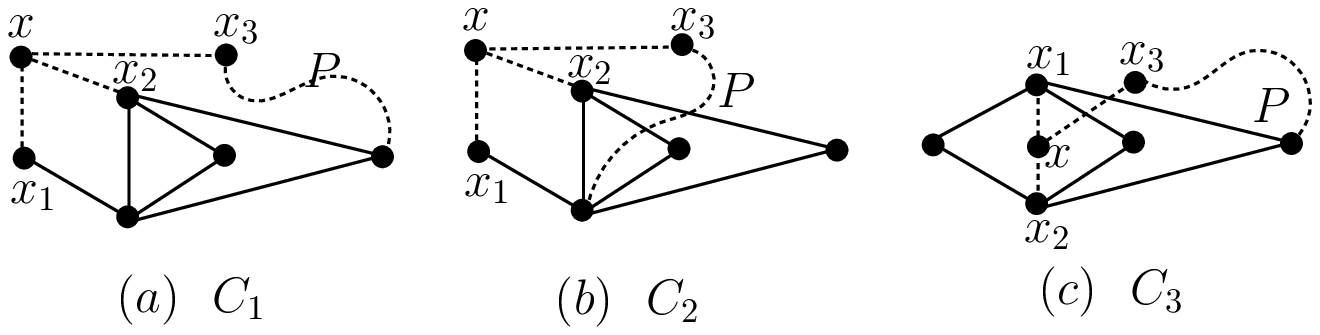}}\\
Figure 8. Graphs for Lemma \ref{lem9}.
\end{center}
\end{figure}

Now we consider $e\notin E(K_4)$. Thus $e\in E(H_r^i)(1\leq i\leq
7)$. We only consider $e\in E(H_r^1)$, and for $e\in E(H_r^i)(2\leq
i\leq 7)$ one can also check that $G\in \mathcal{G}_{n}^{*}$ or
$\overline{\kappa}_3(G)\geq 3$. Since $H_r^1=K_2\vee (r-2)K_1$, we
suppose that $e\in E(K_2\vee (r-2)K_1)(r\geq 3)$. For $r\geq 5$, $G$
must contain one of $\{C_1,C_2,C_3\}$ as its subgraph. One can check
that $\overline{\kappa}_3(G)\geq 3$ by Observation \ref{obs1} (see
Figure 8). For $r=4$, $G\in \mathcal{H}_8^{3}$ or $G\in
\mathcal{H}_8^{4}$ or $G\in \mathcal{H}_{11}^{3}$ or
$\overline{\kappa}_3(G)\geq 3$. For $r=3$, we can obtain $G\in
\mathcal{H}_7^{2}$ or $G\in \mathcal{H}_{10}^{2}$ or
$\overline{\kappa}_3(G)\geq 3$ by Lemma \ref{lem1} and Observation
\ref{obs2}.
\end{proof}

\begin{theorem}
Let $G$ be a connected graph of order $n$ such that
$\overline{\kappa}_3(G)\leq 2$. Then
$$
e(G)\leq \left\{
\begin{array}{cc}
2n-2&if~n=4,\\
2n-3&~~~~~~~if~n\geq 3, n\neq 4.
\end{array}
\right.
$$
with equality if and only if $G\in \mathcal{G}_n^*$.
\end{theorem}

\begin{proof}
We apply induction on $n(n\geq 7)$. For $n=3,4$, it is easy to see
that $\mathcal {G}_{n}^{*}=\{K_n\}$. For $n=5$ or $n=6$, the
assertion holds by Lemmas \ref{lem4} and \ref{lem6}.

Suppose that the assertion holds for graphs of order less than
$n\geq 7$. Now we show that the assertion holds for $n\geq 7$. We
claim that $\delta(G)\leq 3$. Otherwise, $\delta(G)\geq 4$. Let $G'$
be the graph obtained from $G$ by deleting a vertex $x$ such that
$d(x)=\delta(G)$. Then, $2e(G')=2e(G)-2d(x)=2e(G)-2\delta(G)\geq
(n-2)\delta(G)\geq 4(n-2)$. But, by the induction hypothesis,
$2e(G')=2[2(n-1)-3]=4n-10$, a contradiction.

If $\delta(G)=1$, then we let $G'$ be the graph obtained from $G$ by
deleting a pendant vertex. Then by the induction hypothesis,
$e(G)=e(G')+1=2(n-1)-3+1=2n-4<2n-3$.

If $\delta(G)=2$, then we let $G'$ be the graph obtained from $G$ by
deleting a vertex of degree $2$. If $e(G')<2(n-1)-3$, then
$e(G)=e(G')+2<2(n-1)-3+2=2n-3$. If $e(G')=2(n-1)-3$, then
$e(G)=e(G')+2=2(n-1)-3+2=2n-3$. Since $G'\in \mathcal{G}_{n-1}^*$
and $\overline{\kappa}_3(G)\leq 2$, we can obtain $G\in
\mathcal{G}_{n}^*$ by Lemma \ref{lem7}.

Suppose that $\delta(G)=3$. Let $G'$ be the graph obtained from $G$
by deleting a vertex of degree $3$, say $x$. If $e(G')=2(n-1)-3$,
then $G'\in \mathcal{G}_{n-1}^*$. We can get a contradiction by
Lemma \ref{lem8}. If $e(G')<2(n-1)-3$, then $e(G)=e(G')+3\leq
2(n-1)-4+3=2n-3$.

Now we will show that $G\in \mathcal{G}_{n}^*$ for $e(G')=2(n-1)-4$.
Suppose $N_G(x)=\{x_1,x_2,x_3\}$. We have the following two cases to
consider.

\textbf{Case 1. } $G[N_G(x)]$ is not a triangle.

In this case, there exists an edge $x_ix_j\notin E(G)(1\leq i,j\leq
3)$. Let $G''=G'+x_ix_j$. Then we claim that
$\overline{\kappa}_3(G'')\leq 2$. In fact, suppose that
$\overline{\kappa}_3(G'')\geq 3$. Then there exists a $3$-subset
$S\subseteq V(G)$ such that $G''$ contains three internally disjoint
$S$-trees, denoted by $T_1,T_2,T_3$. If $x_ix_j\notin
\bigcup_{i=1}^3E(T_i)$, then $T_1,T_2,T_3$ are $3$ $S$-trees in $G$,
which contradicts $\overline{\kappa}_3(G)\leq 2$.

Assume that $x_ix_j$ belongs to some $S$-tree, without loss of
generality, say $x_ix_j\in E(T_1)$, then $T_1'=(T_1-x_ix_j)\cup
x_ix\cup xx_j$ is an $S$-tree in $G$. Thus, $T_1',T_2,T_3$ are three
internally disjoint $S$-trees in $G$, which implies that
$\overline{\kappa}_3(G)\geq 3$, a contradiction.

Since $e(G'')=e(G')+1=2(n-1)-3$ and $\overline{\kappa}_3(G)\leq 2$,
we have $G''\in \mathcal{G}_{n-1}^*$. Furthermore, $G\in
\mathcal{G}_{n}^*$ by Lemma \ref{lem9}.

\textbf{Case 2. } $G[N_G(x)]$ is a triangle.

Clearly, $G[N_G[x]]$ is a clique of order $4$, where
$N_G[x]=N_G(x)\cup \{x\}$. From Lemma \ref{lem1}, there is no path
connecting any two vertices of $G[N_G[x]]$. So, $G\setminus
E(G[N_G[x]])$ has three connected components except $x$. We denote
them by $G_1,G_2,G_3$ (note that $G_i\neq K_4(i=1,2,3)$). By the
induction hypothesis, $e(G)=\sum_{i=1}^3e(G_i)+6\leq
2\sum_{i=1}^3|G_i|-3=2(n-1)-3<2n-3$.

\end{proof}

\begin{corollary}
$$ f(n;\overline{\kappa}_3\leq 2)=\left\{
\begin{array}{cc}
2n-2&if~n=4,\\
2n-3&~~~~~~~if~n\geq 3, n\neq 4.
\end{array}
\right.
$$
\end{corollary}

Since for $0\leq \ell\leq n-k+\lfloor k/2\rfloor-1$, we have that
$h(n;\overline{\kappa}_k\geq \ell+1)= f(n;\overline{\kappa}_k\leq
\ell)+1$, the following corollary is immediate.

\begin{corollary}
$$ h(n;\overline{\kappa}_3\geq 3)=\left\{
\begin{array}{cc}
2n-1&if~n=4,\\
2n-2&~~~~~~~if~n\geq 3, n\neq 4.
\end{array}
\right.
$$
\end{corollary}

\noindent \textbf{Remark.} Let $n,\ell$ be odd, and $G'$ be a graph
obtained from a $(\ell-3)$-regular graph of order $n-2$ by adding a
maximum matching, and $G=G'\vee K_2$. Then $\delta(G)=\ell$,
$\overline{\kappa}_3(G)\leq \ell$ and
$e(G)=\frac{\ell+2}{2}(n-2)+\frac{1}{2}$.

Otherwise, let $G'$ be a $(\ell-2)$-regular graph of order $n-2$ and
$G=G'\vee K_2$. Then $\delta(G)=\ell$, $\overline{\kappa}_3(G)\leq
\ell$ and $e(G)=\frac{\ell+2}{2}(n-2)+1$.

Therefore,

$$
f(n;\overline{\kappa}_3\leq \ell)\geq \left\{
\begin{array}{cc}
\frac{\ell+2}{2}(n-2)+\frac{1}{2}&for~n,\ell~odd,\\
\frac{\ell+2}{2}(n-2)+1&~otherwise.
\end{array}
\right.
$$

One can see that for $\ell=2$ this bound is the best possible
($f(n;\overline{\kappa}_3\leq 2)=2n-3$). Actually, the graph
constructed for this bound is $K_2\vee (n-2)K_1$, which belongs to
$\mathcal{G}_{n}^*$.


\small
\begin{thebibliography}{11}

\bibitem{Bartfai} P. B\'{a}rtfai, \emph{Solution of a problem proposed by
P. Erd\"{o}s(in Hungarian)}, Mat. Lapok (1960), 175-140.

\bibitem{Bollobas1} B. Bollob\'{a}s, \emph{Extremal Graph Theory},
Acdemic Press, 1978.

\bibitem{Bollobas2} B. Bollob\'{a}s, \emph{On graphs with at most
three independent paths connecting any two vertices}, Studia Sci.
Math. Hungar 1(1966), 137-140.

\bibitem{Bollobas3} B. Bollob\'{a}s, \emph{Cycles and semi-topological configurations,
in: ``Theory and Applications of graphs''(Y. Alavi and D. R. Lick,
eds)} Lecture Notes in Maths 642, Springer 1978, 66-74.

\bibitem{Bondy} J.A. Bondy, U.S.R. Murty, \emph{Graph Theory},
GTM 244, Springer, 2008.

\bibitem{Chartrand} G. Chartrand, S. Kappor, L. Lesniak, D. Lick,
\emph{Generalized connectivity in graphs}, Bull. Bombay Math.
Colloq. 2(1984), 1-6.

\bibitem{Leonard1} J. Leonard, \emph{On a conjecture of Bollob\'{a}s and Edr\"{o}s},
Period. Math. Hungar. 3(1973), 281-284.

\bibitem{Leonard2} J. Leonard, \emph{On graphs with at most four
edge-disjoint paths connecting any two vertices}, J. Combin. Theory
Ser. B 13(1972), 242-250.

\bibitem{Leonard3} J. Leonard, \emph{Graphs with $6$-ways}, Canad. J. Math.
25(1973), 687-692.

\bibitem{LLSun} H. Li, X. Li, Y. Sun, \emph{The generalied $3$-connectivity
of Cartesian product graphs}, Discrete Math. Theor. Comput. Sci.
14(1)(2012), 43-54.

\bibitem{LL} S. Li, X. Li, \emph{ Note on the hardness of generalized
connectivity}, J. Combin. Optim. 24(2012), 389-396.

\bibitem{LLL} S. Li, W. Li, X. Li, \emph{The generalized connectivity of complete
equipartition $3$-partite graphs}, Bull. Malays. Math. Sci. Soc.,
accepted.

\bibitem{LLZ} S. Li, X. Li, W. Zhou, \emph{Sharp bounds for the
generalized connectivity $\kappa_3(G)$}, Discrete Math. 310(2010),
2147-2163.




\bibitem{Mader1} W. Mader, \emph{Ein Extremalproblem des Zusammenhangin endlichen Graphen},
Math. Z. 131(1973), 223-231.

\bibitem{Mader2} W. Mader, \emph{Grad und lokalerZusammenhangs von Graphen},
Math. Ann. 205(1973), 9-11.

\bibitem{Thomassen} B. S{\o}rensen, C. Thomassen , \emph{On
$k$-rails in graphs}, J. Combin. Theory 17(1974), 143-159.

\end{thebibliography}
\end{document}